\documentclass[a4paper,11pt,reqno]{amsart}
\usepackage{amsmath}
\usepackage{amsthm,enumerate}
\usepackage{mathtools}
\usepackage{enumitem}
\usepackage{graphicx}
\usepackage{amssymb}

\usepackage[toc,page]{appendix}

\usepackage[utf8]{inputenc}

\usepackage{color}

\usepackage{url}

\usepackage[text={6.5in,8.6in},centering]{geometry}

\usepackage{srcltx}

\usepackage[T1]{fontenc}

\theoremstyle{plain}
\newtheorem{thm}{Theorem}[section]
\theoremstyle{plain}
\newtheorem{lem}[thm]{Lemma}
\newtheorem{assum}[thm]{Assumption}
\newtheorem{prop}[thm]{Proposition}
\newtheorem{cor}[thm]{Corollary}

\theoremstyle{definition}
\newtheorem{defi}{Definition}[section]
\newtheorem{rem}{Remark}[section]

\newcommand{\R}{\mathbb{R}}
\newcommand{\rn}{\mathbb{R}^{N}}

\newcommand{\hn}{M}
\newcommand{\hnn}{M}
\newcommand{\K}{{\mathbb K}}

\renewcommand{\c}{\mathsf{c}}

\newcommand{\RR}{\mathbb{R}}

\newcommand{\G}{\mathbb{G}^s}
\newcommand{\GM}{\mathbb{G}_{M}^s}

\newcommand{\GH}{\mathbb{G}_{M}^s}

\newcommand{\ka}{\overline{\kappa}}

\newcommand{\dt}{\,{\rm d}t}
\newcommand{\rd}{{\rm d}}

\newcommand{\dg}{\rd\mu_{M}}

\newcommand{\nero }{\color{black}}

\numberwithin{equation}{section} \allowdisplaybreaks

\definecolor{darkblue}{rgb}{0.05, .05, .65}
\definecolor{darkgreen}{rgb}{0.1, .65, .1}
\definecolor{darkred}{rgb}{0.8,0,0}

\begin{document}
 \title[Fractional fast diffusion equations on $\hn$]{Smoothing effects and extinction in finite time for fractional fast diffusions on Riemannian manifolds}
\author{Elvise Berchio}
\address{\hbox{\parbox{5.7in}{\medskip\noindent{Dipartimento di Scienze Matematiche, \\
Politecnico di Torino,\\
        Corso Duca degli Abruzzi 24, 10129 Torino, Italy. \\[3pt]
        \em{E-mail address: }{\tt elvise.berchio@polito.it}}}}}

\author{Matteo Bonforte}
\address{\hbox{\parbox{5.7in}{\medskip\noindent{Departamento de Matem\'aticas,   Universidad Aut\'onoma de Madrid, and \\
ICMAT - Instituto de Ciencias Matem\'{a}ticas, CSIC-UAM-UC3M-UCM,\\
Campus de Cantoblanco, 28049 Madrid, Spain.
 \\[3pt]
        \em{E-mail address: }{\tt matteo.bonforte@uam.es}}}}}

\author{Gabriele Grillo}
\address{\hbox{\parbox{5.7in}{\medskip\noindent{Dipartimento di Matematica,\\
Politecnico di Milano,\\
   Piazza Leonardo da Vinci 32, 20133 Milano, Italy. \\[3pt]
        \em{E-mail address: }{\tt
          gabriele.grillo@polimi.it}}}}}

\date{}



\begin{abstract}
We study nonnegative solutions to the Cauchy problem for the Fractional Fast Diffusion Equation on a suitable class of connected, noncompact Riemannian manifolds. This parabolic equation is both singular and nonlocal: the diffusion is driven by the (spectral) fractional Laplacian on the manifold, while the nonlinearity is a concave power  that  makes the diffusion singular, so that solutions lose mass and may extinguish in finite time. Existence of mild solutions follows by nowadays standard nonlinear semigroups techniques, and we use these solutions as the building blocks for a more general class of so-called weak dual solutions, which allow for data both in the usual $L^1$ space and in a larger weighted space, determined in terms of the fractional Green function. We focus in particular on a priori smoothing estimates (also in weighted $L^p$ spaces) for a quite large class of weak dual solutions. We also show pointwise lower bounds for solutions, showing in particular that solutions have infinite speed of propagation. Finally, we start the study of how solutions extinguish in finite time, providing suitable sharp extinction rates.
\end{abstract}

\maketitle


\normalsize

\section{Introduction}
Let $M$ be an $N$-dimensional ($N\geq 2$) complete, connected, noncompact Riemannian manifold. We study \emph{nonnegative} solutions to the Fractional Fast Diffusion Equation (FFE):
\begin{equation}\label{FFE}
\left \{ \begin{array}{ll}
\partial_t u + (- \Delta_{\hn})^{s} \!\left(u^m\right)= 0 &  (t,x)\in  (0, \infty) \times \hn\,,\\
u(0,x)=u_0(x)\ge0\,, & x\in \hn,
\end{array}
\right.
\end{equation}
where  $(-\Delta_{\hn})^s$ denotes the (spectral) fractional Laplacian on $\hn$, $0 < s < 1$,  $0<m < 1$. The novelty of the present paper lies in the range of exponents for $m$, which corresponds to the so-called \it fractional fast diffusion\rm, as opposed to the case $m>1$ in which the equation is known as the \it fractional porous medium equation\rm. Here, the
 fractional operator can be defined by functional calculus, and also, on a suitable set of functions, by a more explicit formula involving the (minimal) heat kernel of $M$, see formula \eqref{Ks} below, which reminds of the by now classical formula for the Euclidean fractional Laplacian in terms of a singular kernel, see e.g. \cite{CS}. The analogue of \eqref{FFE} in the whole Euclidean space has been the object of intensive research, see e.g. the foundational papers \cite{DQRV1, DQRV2}, and of the later ones \cite{BV3,GMP1, VDQR, Vjems, Vcetr, VV}. Later on, the investigation was extended to the case of different versions of fractional fast diffusions on Euclidean domains, having different probabilistic interpretations and different analytic properties, see e.g. \cite{BF, BII, BSV}. A common feature of some of such works is the use of \it Green function methods\rm, i.e., informally speaking, the idea of studying the \it potential \rm $(-\Delta)^{-s}u(t)$ of a solution $u(t)$, a strategy which we will use here as well.

The study of such equations in the setting of Riemannian manifolds has started very recently, and presently involves only the porous medium case, generalizing several previous results available in the nonfractional ($s=1$) case, see e.g. \cite{BGV, GM, GMP2, GMP, GMV, GMV-MA}. This is a nontrivial topic especially since the operator $(-\Delta)^s$ is in no sense explicit on a Riemannian manifold, and only indirect methods can be used. In fact, the papers \cite{BBGG,BBGM} deal with \eqref{FFE} in the case $m>1$, the first one in the special case of the hyperbolic space, and the second one in much more general setting, namely in the case of manifolds satisfying a Ricci lower bound and a Euclidean-type Faber-Krahn inequality (or, equivalently, a Euclidean-type Nash inequality, or if $N\ge3$ a Euclidean-type Sobolev inequality). The methods used in \cite{BBGG,BBGM} use the $s$-nonparabolicity of the manifold considered, namely the fact that $(-\Delta)^{-s}$ is well-defined through an integral kernel, at least on a suitable set of functions, a fact which is true under the assumptions stated there. Existence of solutions with data which belong to a space which is \it strictly larger \rm than $L^1$, and is naturally associated to the integral kernel defining $(-\Delta)^{-s}$, is then proved, as well as smoothing estimates for the ensuing solutions, i.e. bounds for the $L^\infty$ norm of the solution $u(t)$ in terms of an appropriate norm of the initial datum, which takes into account the fact that initial data can be ``large at infinity'', since their  global integrability is not assumed. It is also worth noticing that the set of data considered is strictly larger than what was known before even when $M=\mathbb R^n$.

Recently, a different geometric situation has been investigated for problem \eqref{FFE} though still in the case $m>1$. In fact, \cite{GMoP} investigates the case of manifolds with nonnegative Ricci curvature, in which in particular the Faber-Krahn, or Nash, or Sobolev inequalities need not hold at least in their Euclidean form, so that the methods of \cite{BBGG,BBGM} fail. Nonetheless, results qualitatively similar to the ones of those latter papers are proved as well, provided $M$ is nonparabolic (i.e. it admits a minimal, positive Green function) and certain uniformity conditions on the volume of Riemannian balls w.r.t. their center hold. It is remarkable that in the resulting smoothing estimates, some quantities related to volume growth of Riemannian balls play an explicit role, similar to the heat kernel bounds proved by Li and Yau \cite{LY} in the same setting. 

The present contribution then aims at starting the analysis of \eqref{FFE} in the fast diffusion case, a setting that seems not to have been considered so far. In fact, existence of solutions is standard, as the results of \cite{BBGM} essentially apply without significant changes. Nonetheless, the proof of smoothing effects needs different methods and tools, which will be provided here. A new phenomenon, w.r.t. the case $m>1$, appears, namely \it extinction in finite time \rm of solution, since diffusion is so fast that, informally speaking, mass is lost at infinity and disappears completely at a suitable time $T>0$. While such phenomenon is well-known and deeply studied in the Euclidean setting, we prove it here in the present much more general setting, proving also that the equation considered gives rise to \it infinite speed of propagation\rm, i.e. solutions corresponding to compactly supported data have instead full support for any $t>0$.

The paper is organised as follows: in Subsections \ref{11}, \ref{12} and \ref{13} below we list our notations, geometric assumptions and related consequences. Section \ref{PMR} is devoted to the statements of our main results. More precisely, in Section \ref{def-wds} we specify the notion of Weak Dual Solutions (WDS) and we discuss their existence and uniqueness; in Section \ref{smooth_sect} we state our smoothing estimates (in $L^p$ and weighted $L^p$ spaces) and we provide sufficient conditions so that solutions extinguish in finite time giving suitable extinction rates; in Section \ref{lower} lower bounds yielding infinite speed of propagation are given. Section \ref{technical} contains a series of crucial estimates and inequalities needed in the proofs of the main results which are instead given in Sections \ref{proof1}, \ref{proof2}, \ref{proof3} and \ref{proof4}. Finally, in the Appendix we provide a short and self-contained proof of the validity of a fractional Euclidean-Type Nash inequality relevant for the present paper, see Proposition \ref{fracnash} below.


\subsection{Geometric assumptions and functional setting}\label{11}


 The manifold $M$ will be required to satisfy certain geometric/analytic properties. The assumptions we make are similar to the ones considered in \cite{BBGM}. In particular, our first and main assumption, which will be required throughout the paper, will be the following one.
\begin{assum}\label{general}
	$M$ is an $N$-dimensional ($N\geq 2$) complete, connected, noncompact Riemannian manifold such that its Ricci curvature is bounded below:
 \begin{equation}\label{ric}
  {\rm Ric}\geq -(N-1)k \quad  \text{ for some } k>0\,.
 \end{equation}
Besides, we require that the following Faber-Krahn inequality holds:
\begin{equation}\label{FK}
 \lambda_1(\Omega)\ge c \, \mu_M(\Omega)^{-\frac2N}
\end{equation}
for a suitable $c>0$, where $\Omega$ is an arbitrary open, relatively compact subset of $M$, $ \mu_M(\Omega)$ denote its measure and $\lambda_1(\Omega)$ is the first eigenvalue of the Laplace-Beltrami operator $-\Delta_M$ with homogeneous Dirichlet boundary conditions on $\partial\Omega$.
\end{assum}

In what follows, for usual $ L^p(\hn) $  spaces ($1\leq p \leq +\infty$), the corresponding norm will typically be written as $ \| \cdot \|_{L^p(\hn)} $, except in some cases where for readability purposes we will adopt the more compact notation $ \| \cdot \|_{p} $. We point out that \eqref{FK} is equivalent to the \emph{Nash inequality}
\begin{equation}\label{Nash}
\left\| f \right\|_2^{1+\frac2N} \le C \left\| f \right\|_1^{\frac2N}  \left\| \nabla f \right\|_2
\end{equation}
 and, when $N\ge3$, to the \emph{Sobolev inequality}
\begin{equation}\label{sob}
\left\| f \right\|_{\frac{2N}{N-2}} \le C \left\| \nabla f \right\|_2 ,
\end{equation}
for all smooth and compactly supported $f$ (see e.g.~\cite[Chapter 8]{H} and \cite{C}).
It is known that the validity of the Euclidean-type Nash inequality \eqref{Nash} implies the validity of its fractional analogue, see \cite{BM} for a precise and general version of this result and Proposition \ref{fracnash} below for the statement relevant for the present paper.

In some of our results, we will need the following stricter assumptions on curvatures.

\begin{assum}\label{ch}
$M$ is an $N$-dimensional Cartan-Hadamard manifold, namely $M$ is complete, simply connected and has everywhere nonpositive sectional curvature.
\end{assum}
Note that if $ M $ is a Cartan-Hadamard manifold, then \eqref{FK} is always true (see again \cite[Chapter 8]{H}), whereas \eqref{ric} should still be required separately. Sometimes, we shall also require a stricter inequality, namely:
\begin{assum}\label{neg}
We require that $M$ is an $ N $-dimensional Cartan-Hadamard manifold and, besides, that
\begin{equation*}
  {\rm sec}(M)\le  -\c \qquad \text{for a given } \c >0\, .
 \end{equation*}
 \end{assum}
 Clearly, the main example we have in mind here is the case of the \it hyperbolic space \rm $\mathbb H^n$. In general, if Assumption \ref{neg} holds, an $L^2$-Poincar\'e inequality holds, namely
 \begin{equation}\label{poi}
\left\| f \right\|_{2} \le C \left\| \nabla f \right\|_2 ,
\end{equation}
for all smooth and compactly supported $f$, implying in particular that  the $L^2$-spectrum of $\Delta$ is bounded away from zero, see e.g. \cite{M}.

\subsection{Fractional Laplacian, fractional potentials and related functional inequalities}\label{12}

%

We now briefly highlight some consequences of the above assumptions in terms of the \emph{fractional Laplacian}, namely the operator $ (-\Delta_M)^s $ defined as the spectral $s$-th power of the  Laplace-Beltrami operator $ -\Delta_M $.  Note that by the spectral theorem it is also given by the explicit formula
$$
 (-\Delta_M)^s v(x) = \int_0^{+\infty} \left(\int_M k_{\hn}(t,x,y)\left(v(y)-v(x)\right)
 \, \dg(y)\right)\, \frac{\rm{d}t}{t^{1+s}}
$$
for a suitable set of functions $v$, see \cite{J}, where $k_{\hn}(t,x,y)$ denotes the heat kernel of $M$. On the other hand, by exploiting \cite[Lemma 2.11]{Caselli2}, it has been proved in \cite[Proposition 6.3]{Caselli} that, in the framework of stochastically complete Riemannian manifolds, the order of integration in the above formula may be changed, therefore we have the following crucial result:

\begin{prop}\label{caselli} (see \cite[Proposition 6.3]{Caselli}) Let $M$ be a complete, stochastically complete Riemannian manifold. For all $v\in C_c^\infty(M)$, for all $x\in M$, one has:
\begin{equation}\label{Ks}
 (-\Delta_{\hn})^s v(x):={\rm P.V.}\int_{\hn}[v(x)-v(y)]\,\K_s(x,y)\dg(y)
\end{equation}
with \[\K_s(x,y):=  c_s\int_0^{+\infty}\frac{k_{\hn}(t,x,y)}{t^{1+s}}\,{\rm{d}t},\]
where $c_s=1/\Gamma(-s)$. If $s<1/2$ the integral is absolutely convergent, hence the principle value is not required.
\end{prop}
It is important to notice that the assumption on the stochastic completeness of $M$ is very mild. For example, it is satisfied when
\[
\textrm{Ric}(x)\ge -Cr(x,x_0)^2
\]
 for suitable $C>0$, $x_0\in M$, and for all $x$ s.t. $r(x,x_0)$ is sufficiently large, see \cite[Theorem 15.4(a)]{Gri}. Therefore, \eqref{Ks} holds in particular under our (much stronger) assumption ${\rm Ric}\geq -(N-1)k.$ This will be fundamental later on and will be used without further comment.

\begin{rem}
{\it
If $M=\mathbb{H}^{N}$ and $N\geq 2$, the validity of \eqref{Ks} was already pointed out in \cite[Theorems 2.4 and 2.5]{BGS} together with the asymptotic behaviors: 
$$\K_s(x,y) \sim r(x,y)^{N-2s} \text{ as } r\rightarrow 0^+\,,  \quad \K_s(x,y)\sim r(x,y)^{-1-s}e^{-(N-1)r(x,y)} \text{ as } r\rightarrow +\infty\,.$$
}
\end{rem}

Furthermore, it is also  well-known that \eqref{FK} implies, for all $x,y\in M$ and $t>0$, and for suitable $C_1, C_2>0$, the following \emph{Gaussian upper bound} on the \emph{heat kernel} $k_{\hn}(t,x,y)$ of $M$:
\begin{equation}\label{gaussian}
k_{\hn}(t,x,y)\le \frac {C_1}{t^\frac N2} \, e^{-C_2\frac{r(x,y)^2}{t}} \, .
\end{equation}
This follows e.g. from \cite[Corollary 15.17]{G} and the subsequent formula (15.49) there.

In particular, thanks to the bound \eqref{gaussian}, we have that $M$ is \emph{$s$-nonparabolic}, in the sense that the integral
\begin{equation}\label{fract}
\GM(x,y):=\int_0^{+\infty}\frac{k_{\hn}(t,x,y)}{t^{1-s}}\,{\rm{d}t}
\end{equation}
is finite for all $x,y\in M$ with $x\not=y$. The function $\GM$ defined above is the \it fractional Green function \rm on $M$, in the sense that $ (-\Delta_M)^s \, \GM(\cdot,y)  = \delta_y $ for every $ y \in M $, where $ \delta_y $ stands for the Dirac delta centered at $y$. Furthermore, from \eqref{gaussian} and \eqref{fract}, one has the Euclidean-type bound
\begin{equation}\label{green-euc}
\GM(x,y)\le \frac C{r(x,y)^{N-2s}} \qquad \forall x,y\in M \, ,
\end{equation}
for some $ C=C(N,c,s)>0 $. Furthermore, in view of the continuity of the map $ x \mapsto k_{\hn}(t,x,y) $ for every fixed $ (t,y) \in \mathbb{R}^+ \times M $, and by virtue of estimate \eqref{gaussian}, also the map $ x \mapsto \GM(x,y) $ turns out to be continuous in $ \hn \setminus \{ y \} $ for every fixed $ y \in M $.

Once the fractional Green function has been introduced we can define, for any sufficiently regular function $ \psi $, its \emph{fractional potential}:
\begin{equation*}
(-\Delta_M)^{-s} \psi (x) : = \int_{\hn} \psi(y) \, \GM(x,y) \, \dg(y) =  \int_0^{+\infty} \left(\int_M\frac{k_{\hn}(t,x,y)}{t^{1-s}}\psi(y)\, \dg(y)\right)\, {\rm{d}t}
\end{equation*}
The spectral theorem ensures that the above operator is indeed the real inverse operator of $ (-\Delta_M)^s $, at least on appropriate subspaces of functions, see e.g. \cite{BBGM} for more details.

\medskip
It will be essential in the sequel to have at our disposal the following functional inequality.
\begin{prop}\label{fracnash}
Suppose that $M$ is a $N$-dimensional manifold supporting inequality \eqref{Nash}.
 Then $M$ supports, for all $s\in(0,1)$, the following fractional Euclidean-type Nash inequality:
\begin{equation}\label{Nash-frac}
\|f\|_2^{1+\frac{2s}N}\le C\|(-\Delta)^\frac s2f\|_2\|f\|_1^{\frac{2s}N}\qquad \forall f\in C_c^{\infty}(M)\,.
\end{equation}

\end{prop}

The proof of Proposition \ref{fracnash} can be found, in a more general form and context, in \cite{BM}. We provide the reader, however, with a short and self-contained proof in the Appendix, since we feel that in the present situation our proof is much simpler and easier to read.

It is well-known that a Nash-type inequality holds, when $N\ge3$, if and only if a Sobolev-type inequality holds, see \cite[Theorems 2.4.2, 2.4.6]{DA}, where such equivalence is proved through equivalent ultracontractive estimates for the semigroup associated to the generator of the quadratic form involved. Hence, Proposition \ref{fracnash} implies the validity of the following Corollary. 

\begin{cor}
Assume that $N\ge3$ and that $M$ supports the Sobolev inequality \eqref{sob}.
Then it supports, for all $s\in (0,1)$, the fractional Sobolev inequality
\begin{equation}\label{Sobolev-frac}
\|f\|_{\frac{2N}{N-2s}}\le C \|(-\Delta)^\frac s2f\|_2\qquad \forall f\in C_c^{\infty}(M).
\end{equation}
\end{cor}
Notice that by density arguments the above inequalities hold for every $f\in H^s(M)$ see e.g. \cite{H}.

%

\subsection{Notation}\label{13} Since we will deal with several multiplying constants, whose exact value is immaterial to our purposes, we will use as much as possible the general symbol $C$. The actual value may therefore change from line to line, without explicit reference. However, when it is significant to specify the dependence of $C$ on suitable parameters, we will write it explicitly while, in some cases, in order to avoid ambiguity, we will use other symbols.

 \section{Main results}\label{PMR}

We initially provide the precise notion of solution to \eqref{FFE} we will work with, and we state our related result regarding existence.

\subsection{Definition of Weak Dual Solutions (WDS), and their existence}\label{def-wds}
We will deal with suitable solutions to \eqref{FFE} starting from initial data that belong to the classical $L^p(\hn)$ space ($1\leq p < +\infty$) or to the following \emph{weighted space}, defined in terms of the fractional Green function:
\begin{equation*}\label{LG}
L^p_{\GM}(\hn) := \left \{ u : \hn \rightarrow \R \text{ measurable} : \ \sup_{x_0\in \hn}   \left\| u \right\|_{L^p_{x_0,\GM}} < +\infty  \right\} ,
\end{equation*}
where, for every fixed $ x_0 \in \hn $, we put
\begin{equation}\label{normLGx0}
\left\| u \right\|_{L^p_{x_0,\GM}}^p  := \int_{B_1(x_0)} \left| u (x) \right|^p \dg(x) + \int_{\hn \setminus B_1(x_0)} \left| u(x) \right|^p \GM(x ,x_0) \, \dg(x) \, .
\end{equation}
The space $ L^p_{x_0,\GM}(\hn) $ is in turn defined as the set of all measurable functions for which the norm in \eqref{normLGx0} is finite.  It is natural to endow $L^p_{\GM}(\hn)$ with the norm
\begin{equation*}
\left\| u \right\|_{L^p_{\GH}}  := \sup_{x_0\in \hn}  \left\| u \right\|_{L^p_{x_0,\GH}} .
\end{equation*}

Thanks to \eqref{green-euc}, $ \GH(x,x_0) \leq C$ for all $x \in \hnn\setminus B_1(x_0)$ and all $x_0\in \hnn$, so that the inclusion $L^p(\hnn) \subseteq L^p_{\GH}(\hnn)$ holds. Moreover, the inclusion $ L^p_{\GM}(\hnn) \subseteq L^p_{x_0,\GH}(\hnn) $ trivially holds by definition. In \cite[Section 4]{BBGM} it is shown that $L^1(\hnn) \subsetneq L^1_{\GM}(\hnn) \subsetneq L^1_{x_0,\GM}(\hnn) $ for all $x_0\in \hnn$, namely the inclusions are \emph{strict}. In addition, admissible decay rates were determined for functions to belong to $L^1_{\GM}(\hnn)$, therefore giving a more noticeable feeling  about how larger these spaces can be compared to $ L^1(M) $. Clearly, these examples can be adapted to the case $1<p< \infty$. To make more explicit the above considerations, we recall that by \cite[Prop. 4.1]{BBGM} we have that $u_0\in L^1_{G^s_M}(M)$ if:

\begin{itemize}
\item $M=\rn$ and $\left| u_0(x) \right| \leq \dfrac{C}{|x|^a}$ for all $|x|\geq R$, for some $C, R>0$ and $a>2s$;
\item $M=\mathbb{H}^{N}$ and $\left|u_0(x)\right|\leq  \dfrac{C}{(r(x,o))^a}$ for all $r(x,o) \geq R$, for some $o\in M$, $C, R>0$ and $a>s$.
\end{itemize}

Clearly in both cases, initial data are allowed to decay \emph{slower} than functions in $L^1(M)$; in fact the allowed powers are \it independent of $N$ in $\mathbb R^N$\rm, while functions in $L^1(\mathbb{H}^{N})$ are expected to decay faster than $e^{-r(x,o)(N-1)}$, whereas power-like decay is instead allowed in the larger space considered.

\it Formally\rm, we can reformulate problem \eqref{FFE} in an equivalent dual form, by means of the inverse operator $(- \Delta_{\hnn})^{-s}$, whose kernel is given by the Green function of $(- \Delta_{\hnn})^{s}$:
\begin{align*}
\left\{\begin{array}{lll}
			 \partial_t \! \left[ (- \Delta_{\hnn})^{-s} u \right] + u^m = 0 & \qquad\mbox{on }(0,+\infty)\times\hnn\,,\\
			u(0,\cdot)=u_0  & \qquad \mbox{in } \hnn\,.
		\end{array}\right.
\end{align*}
Notice that, again formally, this entails that the \it potential \rm of $u(t)$ is \it nonincreasing in time  \rm provided the solution and the Green function are nonnegative.

Next we define a concept of weak solutions suitable for the above formulation, firstly introduced in \cite{BV2,BV1} in the context of bounded Euclidean domains, and generalized to the case of Riemannian manifolds in \cite{BBGG,BBGM}.

\begin{defi}\label{defi_WDS}
 Let $ u_0 \in L^1_{\GM}(\hnn) $, with $ u_0 \ge 0 $. We say that a nonnegative measurable function $ u  $ is a Weak Dual Solution (WDS) to problem~\eqref{FFE} if, for every $T>0$:
\begin{itemize}
	
\item $u \in C^0([0, T]; L^1_{x_0,\GM}(\hnn) )$ for all $ x_0 \in \hnn $;

\smallskip

\item  $u^m \in L^1( (0, T) ; L^{1}_{loc}(\hnn) ) $;

\smallskip

\item $u$ satisfies the identity
\begin{equation}\label{def_eq}
\int_{0}^{T} \int_{\hnn}  \partial_t \psi \,  (- \Delta_{\hnn})^{-s} u \, \dg \, {\rm d}t - \int_{0}^{T} \int_{\hnn} u^m \, \psi \, \dg\, {\rm d}t = 0
\end{equation}
for every test function $\psi \in C^1_c((0,T); L_c^{\infty}(\hn))$;

\smallskip

\item $u(0,\cdot)=u_0$ a.e.~in $\hn$.

\end{itemize}

\end{defi}

Equation \eqref{def_eq} is well defined, see e.g., \cite[Remark 2.1]{BBGM} for more details.
It is possible to construct a (minimal) WDS for any nonnegative initial datum $ u_0\in  L^1_{\GH}(\hnn)$ as a monotone limit of nonnegative semigroup (mild) solutions.

\begin{thm}[Existence of a WDS for data in $L^1_{\GH}$ ]\label{thm-existence}
	Let $ \hn $ satisfy Assumption \ref{general} and let $u_0 $ be any nonnegative initial datum such that $u_0\in  L^1_{\GH}(\hnn)$. Then there exists a weak dual solution to problem~\eqref{FFE}, in the sense of Definition \ref{defi_WDS}.
\end{thm}

Moreover, as in \cite{BBGM,BV1}, it is possible to show that within this subclass, solutions are unique:
\begin{cor}[Uniqueness of limit WDS]
The WDS $u$ constructed in Theorem \ref{thm-existence} as a monotone limits of mild $ L^1(\hnn) \cap L^\infty(\hnn)$ solutions, does not depend on the particular choice of the monotone approximating sequence of initial data.
\end{cor}

We shall \it not \rm provide an explicit proof, of the above results, since they follow exactly along the same lines given in \cite{BBGM}. We just remark what follows:

\begin{itemize}

\item As concerns existence of mild solutions and weak dual solutions, nowadays classical nonlinear semigroup theory can be applied in the present case, to show existence of mild solutions, often called also semigroup or gradient-flow solutions. These are the building blocks for the existence theorem of WDS.

For nonnegative data in $ L^1(\hnn) \cap L^\infty(\hnn)$, a version of the celebrated Crandall-Liggett theorem applies also in this case (notice that the nonlinearity is concave, and the equation can be very singular), through the theory later developed in \cite{BCP}, see also \cite{CP}, that applies in the present setting with minor modifications. This allows to construct nonnegative $L^1$-mild solutions  in particular enjoying the time monotonicity property:
	\begin{equation}\label{mon-est.OLD}
	\mbox{the map} \quad t \mapsto t^{\frac{1}{m-1}} u(t,x) \quad \mbox{is (essentially) nonincreasing for a.e.~$x \in \hn $}
	\end{equation}
	and the $ L^p(M) $-nonexpansivity property:
	\begin{equation}\label{p-decay}
	\left\| u(t) \right\|_{L^p(\hn)} \leq \left\| u_0 \right\|_{L^p(\hn)} \qquad \mbox{for all $t\ge 0$ and all $1 \le p \leq \infty \, .$}
	\end{equation}
Finally, since $L^1(M)$ is included (with continuity) in $ L^1_{x_0,\GM}(\hn) $ for every $x_0 \in \hn$ these mild solutions turn out to be also WDS, see in particular  \cite[Section 5]{BBGM}.

\item A solid alternative is provided by the Brezis-Komura Theorem, that allows to build mild solutions in the Hilbert space $H^{-s}$, the dual of $H^s$. In this case existence and uniqueness of -possibly sign changing- mild solutions with data in $H^{-s}$ is also ensured, see Section 5.2 of \cite{BBGG}, where the case $m>1$ is analyzed but that indeed holds for all $m>0$. To the best of our knowledge, this is the biggest class in which existence and uniqueness holds for possibly sign changing solutions. This has been described with more details in Section 5 of both \cite{BBGM} and \cite{BBGG}, for the Porous Medium case $m>1$, but the same theory works indeed for all $m>0$.

\item For nonnegative initial data, to the best of our knowledge, the biggest class of data for which existence is ensured is precisely the one we find here, namely WDS with data in $L^1_{\GH}(\hnn)$. We notice that the proof of \cite[Theorem 2.4]{BBGM} works also in this case with minor changes: one first approximates the nonnegative initial datum  $u_0$, that merely belongs to  $ L^1_{\GM}(\hn) $, with a sequence of ``truncated'' data $ u_{0,n} \in L^1(M) \cap L^\infty(M) $ yielding approximate WDS which are finally extended to general WDS, by means of a standard limiting process relying on the stability property stated in Proposition \ref{LpPhi stability} below.
Actually, here the only advantage of the concave nonlinearity is represented by the fact that $u\in L^1_{\rm loc}$ implies $u^m\in L^1_{\rm loc}$, while in the case $m>1$ this is not true, and $L^1-L^\infty$ smoothing effects must enter into play. In the present case, no smoothing effects are needed in the proof of existence, which follows the same lines as in Section 6.1 of \cite{BBGM} and it is actually simpler in the present case. It is worth noticing that when $m$ is small, in the sense that $m<m_c:=\frac{N-2s}{N}$, solutions corresponding to $L^1$ data can be unbounded, as it well-known in the Euclidean case, see \cite{V2}, and can be easily proved as well in our case.

\end{itemize}

\subsection{Statement of the main results concerning smoothing effects}\label{smooth_sect}

Let us define the exponents:
\[
m_c:=\frac{N-2s}{N}\qquad\mbox{and}\qquad p_c:=\frac{N(1-m)}{2s}\,,
\]
needed to state our main results about $L^p-L^\infty$ smoothing estimates, with and without weights. The two exponents are related by the fact that $p_c \geq 1$ if and only if $m\in(0,m_c]$.

\begin{thm}[$L^p-L^\infty$ smoothing]\label{Thm.Smoothing} Let $ \hn $ satisfy Assumptions \ref{general} and  \ref{ch}. Furthermore, let $N>2s$, $m\in (0,1)$ and let $u$ be the nonnegative WDS of \eqref{FFE}, constructed in Theorem \ref{thm-existence} and corresponding to the initial datum $u_0\in  L^1_{\GH}(\hnn) \cap L^p(\hnn)$ with $p_c<p < +\infty$ if $m\in(0,m_c]$ or $1 \leq p < +\infty$ if $m\in(m_c,1)$. Then, for every $t> 0$ we have
		\begin{equation}\label{Thm.Smoothing.ineq}
	 	\|u(t)\|_\infty\leq \ka\;\frac{\|u_0\|^{2sp\vartheta_{p}}_p}{t^{N\vartheta_{p}}}\qquad\qquad\mbox{with}\quad \vartheta_{p}=\frac{1}{2sp-N(1-m)}>0
		\end{equation}
	where $0<\ka$ only depends on $N,m,s,p$. If only  Assumption \ref{general} is satisfied, then \eqref{Thm.Smoothing.ineq} holds provided that $0<t\leq \|u_0\|^{1-m}_{p} $ while for all $t\geq \|u_0\|^{1-m}_{p} $ we have
		  	\begin{equation}\label{Thm.Smoothing.ineqweight0}
		  \|u(t)\|_\infty \leq \ka_0\;\frac{\|u_0\|_{p}^{\frac{p}{p+m-1}}}{t^{\frac{1}{p+m-1}}}	\,,
		  \end{equation}
		  where $0<\ka_0$ only depends on $N,m,s,p$.
	

	\end{thm}

 We notice that, for those values of $p,m, N$ considered in Theorem \ref{Thm.Smoothing.ineq}, $N\vartheta_{p}>\frac{1}{p+m-1}$, therefore \eqref{Thm.Smoothing.ineq} yields a better bound than \eqref{Thm.Smoothing.ineqweight0} for $t$ large. This is somehow in accordance with the statement of Theorem \ref{ext} below where we prove that when Assumption \ref{neg} holds (and, in turn, Assumption \ref{ch} as well), then solutions extinguish in finite time. 

\begin{thm}[Extinction Time]\label{ext}
Let $ \hn $ satisfy Assumptions \ref{general} and \ref{neg}.  Let $u$ be a nonnegative WDS corresponding to the initial datum $u_0\in  L^1_{\GH}(\hnn) \cap L^p(\hnn)$ with $p>1$ and $p\geq p_c$ (namely, $p_c\leq p <+\infty$ if $0<m<m_c$ and $1<p<+\infty$ for $m\in [m_c,1)$). Then $u$ extinguishes at a finite time $T=T(u_0)$ and
	for every $0 \le t_0\leq t \le T$ we have
	\begin{equation}\label{Lp estimate}
		c_p(T-t)\le \|u(t)\|_p^{1-m} \le \|u(t_0)\|_p^{1-m} - c_p (t-t_0)\,,
	\end{equation}
	where  $c_p>0$  only depends on $p,m,s,N$ and $c_p \rightarrow 0$ as $p \rightarrow 1^+$. When only  Assumption \ref{general} is satisfied the statement holds under the restriction that $p= p_c$ and $0<m<m_c$.
\end{thm}

Some comments are in order about the sharpness of the extinction rate provided by Theorem \ref{ext}. To this aim, let $H^{-s}(\hn)$ denote the dual of $H^s(\hn)$, with the Hilbertian norm given by
\begin{equation*}
\|u\|_{H^{-s}(\hn)}^2=\int_{\hn} u(-\Delta_{\hn})^{-s} u\, \dg = \int_{\hn} \left|(-\Delta_{\hn})^{-s/2} u\right|^2 \dg
\end{equation*}
and the corresponding scalar product. For all $u\in L^{1+m}(\hn)\cap H^{-s}(\hn) \setminus\{0\}$, one may define the ``Dual'' Nonlinear Rayleigh Quotient as follows
\begin{align*}
	\mathcal{Q}^*[u]:= \frac{\|u\|_{1+m}^{1+m}}{\|u\|_{H^{-s}}^{1+m}}\,.
\end{align*}
The  same computations as in \cite[Proposition 8.3]{BF23} allow to check that $\frac{\rd}{\dt}\mathcal{Q}^*[u(t)]\leq 0$ along the flow and in turn to prove:

\begin{prop}[Sharp $L^{1+m}$ extinction rate]
	Let $ \hn $ satisfy Assumptions \ref{general} and \ref{neg} and let $u$ be a nonnegative WDS with initial data $u_0\in  L^1_{\GH} (M)\cap L^{1+m}(M)\cap H^{-s}(\hn)$. If there exists  an  extinction time $T>0$, then
	\begin{align*}
	\|u(t)\|_{H^{-s}}\leq c_1\;\mathcal{Q}^*[u_0]^{\frac{1}{1-m}} (T-t)^{\frac{1}{1-m}}\qquad\mbox{for every }\;0\leq t<T,
	\end{align*}
	with $c_1=(1-m)^{\frac{1}{1-m}}$. Moreover, there  exists  $c_*>0$ depending only on $m$ such that for all $t\in [0,T]$
	\begin{equation}\label{Lm+1 estimate}
		\|u(t)\|_{1+m}^{1+m} \le c_*^{1+m}\, \mathcal{Q}^*[u_0]^{\frac{2}{1-m}}
		\left\{\begin{array}{lll}
			  \frac{T^{\frac{2}{1-m}}}{t}  & \qquad \mbox{when } 0< t< \frac{T}{3},\\
			  (T-t)^{\frac{1+m}{1-m}} & \qquad \mbox{when }  \frac{T}{3}\leq t \leq T\,.
		\end{array}\right.
	\end{equation}
\end{prop}
When $p=1+m$ and $m>m_s:=\frac{N-2s}{2s+N}$, then $1+m> p_c$ and the left hand side of \eqref{Lp estimate} combined with \eqref{Lm+1 estimate} yield
 $$c_p^{\frac{1+m}{1-m}} (T-t)^{\frac{1+m}{1-m}} \leq \|u(t)\|_{1+m}^{1+m} \le c_*^{1+m}\, \mathcal{Q}^*[u_0]^{\frac{2}{1-m}} (T-t)^{\frac{1+m}{1-m}}  \qquad \mbox{when } \frac{T}{3}\leq t \leq T\,,$$
showing that $(T-t)^{\frac{1+m}{1-m}}$  is the sharp extinction decay.  If $m=m_s,$ then $m+1=p_c$, the bound still holds but, since $m_s<m_c$, Theorem \ref{Thm.Smoothing} does not apply and cannot guarantee a priori that solutions with data in $L^{1+m}$ are bounded. In the Euclidean case when $s=1$, i.e. for the ``classical'' fast diffusion posed on the whole space $\mathbb{R}^N$, there exist Very Singular Solutions (VSS, they exist only when $m< m_c$) with an explicit separate-variables form: they are not bounded and still they extinguish at a finite time with the above rate, see for instance the monograph \cite{V2}. In the fractional Euclidean case $s<1$, such VSS still exist and are unbounded and still extinguish at the same rate, but they do not have an explicit form (in the  spacial variables), see \cite{Vjems, VV}.

When enlarging the class of allowed initial data, i.e.~when dealing with the space $L^p_{\GH}(\hnn)$ in place of $ L^p(\hnn) $,  we obtain the following $L^p_{\GH}$-$L^{\infty}$ smoothing estimates.

\begin{thm}[$L^p_{\GH}-L^\infty$ smoothing]	\label{Thm.SmoothingPhi} Let $ \hn $ satisfy Assumption \ref{general}. Furthermore, let $N>2s$ and $m\in (0,1)$. Let $u$ be the nonnegative WDS of \eqref{FFE}, constructed in Theorem \ref{thm-existence} and corresponding to the initial datum $u_0\in  L^1_{\GH}\cap L^p_{\GH}(\hnn)$ with $1 \leq p < +\infty$ if $m\in(m_c,1)$ or $p_c<p < +\infty$ if $m\in(0,m_c]$. Then, for all $0<t\leq \|u_0\|^{1-m}_{L^p_{\GH}} $ we have
		\begin{equation}\label{Thm.Smoothing.ineqweight}
		  \|u(t)\|_\infty \leq \ka_1\;\frac{\|u_0\|_{L^p_{\GH}}^{2sp\,\vartheta_{p} }}{t^{N\vartheta_{p}}}	
		  \end{equation}
		  while for all $t\geq \|u_0\|^{1-m}_{L^p_{\GH}} $ we have
		  	\begin{equation}\label{Thm.Smoothing.ineqweight2}
		  \|u(t)\|_\infty \leq \ka_2\;\frac{\|u_0\|_{L^p_{\GH}}^{\frac{p}{p+m-1}}}{t^{\frac{1}{p+m-1}}}	\,,
		  \end{equation}
		  where $0<\ka_1, \ka_2$ only depend on $N,m,s,p$.

		
	\end{thm}

\subsection{Lower bound and infinite speed of propagation}\label{lower}

We now state our result concerning \it lower bounds \rm for the solutions considered. It reads as follows.
 \nero
 \begin{thm}[Infinite speed of propagation and pointwise lower bound]\label{lowerth}
Let $ \hn $ satisfy Assumptions \ref{general}. Furthermore, let $u$ be a nonnegative nontrivial WDS corresponding to the nonnegative initial datum $u_0\in  L^1_{\GH}(\hnn)$ and assume that there exists an extinction time $T=T(u_0)$ of $u$. For all $0<t<T$, $x\in M$, there holds
\begin{equation}\label{pointwise0}
u^m(t,x)\geq
C  \frac{t^{\frac{m}{1-m}}}{(1-m)T^{\frac{1}{1-m}}} \, \left\| u (t) \right\|_{L^1_{x,\GM}}
	 \end{equation}
  for some $C=C(N,k,c,s)$. Hence, solutions have infinite speed of propagation, in the sense that solutions corresponding to data with compact support become instantaneously supported in the whole $M$. In particular, if also Assumptions \ref{ch}  is satisfied and $u_0\in L^{1}(\hnn)\cap L^{p}(\hnn)$ with $p>\frac{N}{2s}$, there exist $0<t_0(u_0)\leq T(u_0)$ and $C_m=C_m(N,k,c,s,m,p)>0$ such that
\begin{equation}\label{pointwise}
u^m(t,x)\geq C_m\frac{t^{\frac{m}{1-m}} }{T^{\frac{1}{1-m}}}\, \left(1\wedge r(x_0,x)^{N-2s}\right) \GM(x,x_0)\,\|u_0\|_{L^1(B_R(x_0))}
 \end{equation}
for all $t\in[0,t_0(u_0)]$ and all $x\in M\setminus \{x_0\}$,  where $x_0$ is a point such that $\|u_0\|_{L^1(B_R(x_0))}>0$ for some $R>0$.
\end{thm}

\begin{rem} We comment in first place that a class of data s.t. the corresponding solution vanishes in finite time is given in Theorem \ref{ext}, in particular one can take $u_0\in  L^1_{\GH}(\hnn) \cap L^p(\hnn)$ with $p>1$ and $p\geq p_c$ (namely, $p_c\leq p <+\infty$ if $0<m<m_c$ and $1<p<+\infty$ for $m\in [m_c,1)$). However, the core of the above result is \it instantaneous positivity\rm, which follows for \it general \rm data $u_0\in L^1_{\GH}(\hnn)$ by estimating from below the corresponding solution with the one associated to $(u_0\wedge n)\chi_{B_n(x_0)}$, $n\in \mathbb N$, which of course has an extinction time. In particular, formula \eqref{pointwise} is significant since it both proves that solutions corresponding even to compactly supported data become \it instantaneously strictly positive anywhere \rm and, besides, since it provides an explicit \it pointwise spacial lower \rm bound, in terms of the fractional Green function, for such solutions.\end{rem}

\section{Technical results} \label{technical}
\subsection{Fractional Green function estimates}
In this section we collect some estimates for fractional Green functions and potentials that have been proved in \cite{BBGM} and which will be needed in our proofs. We have already remarked that, under the \emph{Assumption \ref{general}}, the Euclidean-type bound \eqref{green-euc} holds. As a consequence, by exploiting the Bishop-Gromov Theorem (see e.g., the proof of \cite[formula (6.3)]{BBGM}), for all $1\leq q<\frac{N}{N-2s}$ direct computations in radial coordinates show that
\begin{equation}
\label{Hyp.Green.HN1}
\int_{B_R(y)} \left(\GH(x,y) \right)^q \,  \dg(x)  \le C R^{N-q(N-2s)} \qquad \text{for all $ 0 < R \le 1 $} \,,
\end{equation}
where $C=C(N,k,c,s,q)$. Under Assumption \ref{ch}, by arguing as in the proof of \cite[formula (3.7)]{BBGM}, the above estimate can be extended to all $R>0$, namely we have
\begin{equation}\label{Hyp.Green.HN3}
\int_{B_R(y)} \left(\GH(x,y) \right)^q \,  \dg(x)  \le C R^{N-q(N-2s)} \qquad \text{for all $ R>0 $} \, ,
\end{equation}
for all $1\leq q<\frac{N}{N-2s}$ and for a suitable $C=C(N,c,q)>0$.\par
In the proof of the infinite speed of propagation we will instead need the following lower bound
\begin{equation}\label{green-est-low}
\GM(x, x_0) \ge \frac{C}{r(x,x_0)^{N-2s}} \qquad \text{for all $ x \in \hn $: }  r(x,x_0) < 2
\end{equation}
for some $C=C(N,k,s)>0$, which holds under Assumption \ref{general} and has been proved in  \cite[formula (3.25)]{BBGM} by combining \eqref{fract} with suitable lower bounds of the heat kernel.
\par
We now recall from \cite[Lemma 3.2]{BBGM} a comparison result between potentials of bounded compactly supported functions with the Green function. This will be a crucial ingredient in the proof of the stability result stated in Lemma \ref{LpPhi stability} below.
 \begin{lem}\label{lem.Phi.estimates}
Let $ \hn $ satisfy Assumption \ref{general}. Let $ \psi \in \mbox{L}_c^{\infty}(\hn)$ be a nonnegative and nontrivial function such that supp$( \psi)\subseteq B_{\sigma}(x_0)$ for some $0<\sigma<1$ and $x_0 \in M$. Then there exist two constants $\underline C=\underline C(N,k,c,s)>0$ and $ \overline C= \overline C(N,k,c,s)>0$ such that
\begin{equation*}
\underline C \left\| \psi \right\|_1 \left(1\wedge r(x_0,x)^{N-2s}\right) \GH(x, x_0) \leq  (-\Delta_{\hnn})^{-s} \psi  (x) \leq \overline C\, \|\psi\|_{\infty} \, \sigma^N  \, \GH(x,x_0)  \quad \forall x\in \hnn \setminus \{ x_0 \} \, .
\end{equation*}
 \end{lem}

\subsection{Key inequalities}
In this section we collect a series of technical inequalities that will be exploited in the proofs of the main results.
The same proofs of \cite[Lemmas 6.1 and 6.2]{BII} in the euclidean setting (and which relies on the fact that the singular kernel is nonnegative) yields the following two lemmas. We rewrite the (short) proof of the first of such results in order to stress the importance of Proposition \ref{caselli} for the present results. The proof of Lemma \ref{S-V} also depends crucially on the representation formula given in Proposition \ref{caselli} as well.
\begin{lem}[Kato's inequality]
	Let $f\in{C}^0(\RR)$ be a convex function with $f(0)\leq 0$. Then, if $v$ and $(- \Delta_{\hnn})^{s} v\in L^1_{\rm loc(\hnn)}$, Kato's inequality holds in the sense of distributions:
	\begin{equation*}
		(- \Delta_{\hnn})^{s} f(v)\leq f^\prime(v)\;(- \Delta_{\hnn})^{s} v\,.
	\end{equation*}
\end{lem}
\begin{proof} Since $\K_s(x,y)\geq 0$ in $\hnn\times \hnn$ and $f$ is convex, by \eqref{Ks} we have that
	\begin{align*}
		(- \Delta_{\hnn})^{s} f(v(x))&=\int_{\hnn}\left[f(v(x))-f(v(y))\right]\K_s(x,y)\dg(y)\\
		&\leq \int_{\hnn}f^\prime(v(x))\left[v(x)-v(y)\right]\K_s(x,y)\dg(y)\\
		&=f^\prime(v(x))(- \Delta_{\hnn})^{s} f(v(x))\,.
	\end{align*}
	\end{proof}
	
\begin{lem}[Stroock-Varopoulus inequality]\label{S-V}
	For any $q>1$
	\begin{equation}\label{Stroo}
		\int_{\hnn} v^{q-1}(- \Delta_{\hnn})^{s} v \, \dg(y) \geq \frac{4(q-1)}{q^2}\int_{\hnn}\left| ((- \Delta_{\hnn})^{s})^{1/2}v^{q/2}\right|^2 \dg(y)
	\end{equation}
	for all $v\in L^q(\Omega)$ such that $((- \Delta_{\hnn})^{s})^{1/2}v\in L^q(\Omega)$.
\end{lem}
By exploiting Kato's inequality we obtain the two statements below.
\begin{prop}[Fundamental upper bounds]
Let $p\ge1$ and let $u$ be a nonnegative mild solution with $u_0\in L^1(\hnn)\cap L^\infty(\hnn)$, then for all $x_0\in\hnn$ and $0\leq t_0\leq \tau< t \leq t_1$ we have
\begin{equation}\label{pointwise Lp estimate}
	u^{p+m-1}(t,x_0)\leq \c_{p,m}\;\frac{t_1^{\frac{p+m-1}{1-m}}}{(t-\tau)^{\frac{p}{1-m}}}\;\int_{\hnn}u(\tau,x)^p\;\GH(x_0,x)\dg(x)\,,
	\end{equation}
where $\c_{p,m}=\frac{p+m-1}{m(1-m)}$.
\end{prop}
\begin{proof} For $p>1$, we multiply the equation in \eqref{FFE} by $p u^{p-1}\,\GH$ and integrate on $[0,T]\times\hnn$. On one hand, we obtain
\begin{equation*}
  \int_{\hnn}\int_{t_0}^{t_1}\!p u^{p-1}(t,x)\;\partial_t u(t,x)\; \GH(x_0,x)\dt\dg(x)=\!\int_{\hnn}\!\!\left(u(t_1,x)^p-u(t_0,x)^p\right)\GH(x_0,x)\dg(x).
\end{equation*}
On the other hand, by Kato's inequality with $v=u^m$ and $f(v)=\tfrac{m}{p+m-1}v^{\frac{p+m-1}{m}}$, we get
\begin{align*}
&-p\int_{t_0}^{t_1}\int_{\hnn} u^{p-1} (- \Delta_{\hnn})^{s} \left(u^m(t,x)\right)\GH(x_0,x)\dg(x)\dt\\
  \le&-\frac{pm}{p+m-1}\int_{t_0}^{t_1}\int_{\hnn} (- \Delta_{\hnn})^{s} (u^{p+m-1}(t,x))\GH(x_0,x)\dg(x)\dt\\
  =&-\frac{pm}{p+m-1}\int_{t_0}^{t_1}u^{p+m-1}(t,x_0)\dt \leq -\frac{m(1-m)}{p+m-1}  \left( \frac{t_1^{\frac{p}{1-m}}-t_0^{\frac{p}{1-m}}}{t_1^{\frac{p+m-1}{1-m}}} \right) u^{p+m-1}(t_1,x_0)  \,,
\end{align*}
where in the last step we have used the time monotonicity \eqref{mon-est.OLD}, namely that $\left(\frac{{t}}{t_1}\right)^{\frac{1}{1-m}}u(t_1)\leq u({t})$. Combining the above estimates and taking $t_1=t$, $t_0=\tau$, we get
$$u^{p+m-1}(t,x_0)\leq \c_{p,m}\;\frac{t^{\frac{p+m-1}{1-m}}}{(t_1)^{\frac{p}{1-m}}-(\tau)^{\frac{p}{1-m}}}\;\int_{\hnn}u(\tau,x)^p\;\GH(x_0,x)\dg(x)$$
 and we obtain \eqref{pointwise Lp estimate} by noticing that $t^{\frac{p+m-1}{1-m}}\leq t_1^{\frac{p+m-1}{1-m}}$ and $(t_1)^{\frac{p}{1-m}}-(\tau)^{\frac{p}{1-m}}\geq (t-\tau)^{\frac{p}{1-m}}$ for $0\leq t_0\leq \tau< t \leq t_1$. The case $p=1$ follows similarly but by testing \eqref{def_eq} with a smooth approximation of $\delta_{x_0}\chi_{[t_0, t_1]}$, see \cite[Lemma 3.4]{BII}.
\hfill
\end{proof}

\begin{lem}[$L^p_{x_0,\GH}$-stability]\label{LpPhi stability}
	Let $p\ge1$ and let $u$ be a nonnegative mild solution with $u_0\in L^1(\hnn)\cap L^\infty(\hnn)$, then
	\begin{align}\label{stab_ineq}
	\|u(t)\|_{L^p_{x_0,\GH}}\leq\|u_0\|_{L^p_{x_0,\GH}}\qquad \text{for all } t\geq 0 \text{ and all }x_0\in \hnn\,\,.
	\end{align}
\end{lem}

\begin{proof}
	For $0 <t_0<t_1 $, let us multiply the equation in \eqref{FFE} pointwise by $pu^{p-1} (- \Delta_{\hn})^{-s} \psi(x) \chi_{[t_0, t_1]}(t)$ for some $\psi \in L^{\infty}_{c}(\hnn)$ (this can be replaced by a smooth approximation, see \cite[Lemma 6.4]{BII}) and integrate:
	\begin{align*}
	&\int_{t_0}^{t_1}\int_{\hnn}pu^{p-1}(- \Delta_{\hnn})^{s} u^m\; (- \Delta_{\hn})^{-s} \psi(x) \dg(x)\dt \\
	&=\int_{t_0}^{t_1}\int_{\hnn}pu^{p-1}\partial_t u\,(- \Delta_{\hn})^{-s} \psi(x) \dg(x)\dt\\
	&=\int_{0}^T\int_{\hnn}\partial_t\left(u^p\,(- \Delta_{\hn})^{-s} \psi(x) \right)\chi_{[t_0, t_1]}\dg(x)\dt\\&=\int_{\hnn}u^p(t_1)(- \Delta_{\hn})^{-s} \psi(x)\dg(x)-\int_{\hnn}u^p(t_0)(- \Delta_{\hn})^{-s} \psi(x)\dg(x)\,.
	\end{align*}
On the other hand, by Kato's inequality (with $v=u^m$ and $f(v)=\tfrac{m}{p+m-1}v^{\frac{p+m-1}{m}}$) we get
	\begin{align*}
	&-\int_{t_0}^{t_1}\int_{\hnn}pu^{p-1}(- \Delta_{\hnn})^{s} u^m\; (- \Delta_{\hn})^{-s} \psi(x) \dg(x)\dt \\&\leq-\frac{p m}{p+m-1}\int_{t_0}^{t_1}\int_{\hnn}(- \Delta_{\hnn})^{s} u^{p+m-1}(- \Delta_{\hn})^{-s} \psi(x) \dg(x)\dt\\
	&=-\frac{p m}{p+m-1}\int_{t_0}^{t_1} \int_{\hnn} u^{p+m-1} \psi(x) \dg(x) \dt\leq 0\,.
	\end{align*}	
	Whence, the following inequality holds:
\begin{align}\label{limit-1}
\int_{\hnn}u^p(t_1)(- \Delta_{\hn})^{-s} \psi(x)\dg(x)\leq \int_{\hnn}u^p(t_0)(- \Delta_{\hn})^{-s} \psi(x)\dg(x) \,,
\end{align}
for every nonnegative $\psi \in L^{\infty}_{c}(\hnn)$, and all $ 0 <t_0<t_1 $.  Then \eqref{stab_ineq} follows from \eqref{limit-1} by noticing that, in view of Lemma \ref{lem.Phi.estimates}, there exist a nonnegative and nontrivial function $ \psi \in \mbox{L}_c^{\infty}(\hn)$ and two constants $c_1, c_2>0$, depending only on $ N,s,k,c  $ (in particular independent of $ x_0 $), such that for all nonnegative $u\in L^p_{\GM}(\hnn)$ one has
\begin{equation*}
c_1 \, \int_M u^p \, (- \Delta_{\hn})^{-s} \psi \,  {\rm d}\mu_M\le\left\| u\right\|_{L^p_{x_0,\GM}}^p \le c_2 \, \int_M u^p \, (- \Delta_{\hn})^{-s} \psi \,  {\rm d}\mu_M \, .
\end{equation*}
See \cite[Lemma 3.3]{BBGM} for a detailed proof in the case $p=1$.

\end{proof}

\section{Proof of Theorem \ref{Thm.Smoothing}}\label{proof1}

We first give the proof for $ \hn $ satisfying Assumptions \ref{general} and \ref{ch}.

\noindent It is enough to prove this theorem for bounded nonnegative mild solutions $u \in L^1(\hnn)\cap L^\infty(\hnn)$. Indeed, we can approximate the WDS with initial datum $u_0\in  L^1_{\GH}(\hnn) \cap L^p(\hnn)$ by means of mild solutions  $u_n(t)$ starting at $u_{0,n}=\min\lbrace u_0,n\rbrace  \chi_{B_n(o)} \in L^1(\hnn)\cap L^\infty(\hnn)$, see Step 3 for further details.

\noindent$\bullet~$\textsc{Step 1. }\textit{Fundamental pointwise estimate and De Giorgi Lemma.} Let $u$ be a nonnegative  mild solution with $u_0\in L^1(\hnn)\cap L^\infty(\hnn)$ and let $\varepsilon\in\left(0,\tfrac{2sp-N(1-m)}{N}\right)\subset(0,p+m-1)$. Then, for all $0\le t_0 < t_1$, the following estimate holds:
\begin{equation}\label{Green.1.0}\begin{split}
&\|u(t_1)\|_\infty^{p+m-1}
 \leq \\
 &\;2^{\frac{p+m-1-\varepsilon}{\varepsilon}} c\left(\frac{p(p+m-1)}{(1-m)\varepsilon}, \lambda \right) \left[\c_{p,m}\frac{t_1^{\frac{p+m-1}{1-m}}}{(t-t_0)^{\frac{p}{1-m}}} \sup\limits_{\substack{\tau\in[t_0,t_1]\\x_0\in\hnn}} \int_{B_R(x_0)}\!\!\!\!\!\!\!\!u^{1-m+\varepsilon}(\tau,x)\,\G(x_0,x)\dg(x)\right]^{\frac{p+m-1}{\varepsilon}} \\
 & + c\left(\frac{p}{1-m}, \lambda \right) \c_{p,m}\;\frac{t_1^{\frac{p+m-1}{1-m}}}{(t-t_0)^{\frac{p}{1-m}}}\; \sup\limits_{\substack{\tau\in[t_0,t_1]\\x_0\in\hnn}}\int_{\hnn\setminus B_R(x_0)}u^p(\tau,x)\;\G(x_0,x)\dg(x)\,,
\end{split}\end{equation}
where
	\begin{equation*}
		c(\alpha, \lambda)=\frac{1}{\left(1-\lambda\right)^{\alpha}\left(1-\frac{1}{2\lambda^{\alpha}}\right)}\qquad\mbox{for any}\qquad\lambda\in(2^{-\frac{1}{\alpha}}, 1)\,.
	\end{equation*}

In order to prove \eqref{Green.1.0}, we apply the upper bound \eqref{pointwise Lp estimate} which gives that for all $x_0\in\hnn$ and $0\leq t_0\leq \tau< t \leq t_1$,	
\begin{equation*}
u^{p+m-1}(t,x_0)\leq \c_{p,m}\;\frac{t_1^{\frac{p+m-1}{1-m}}}{(t-\tau)^{\frac{p}{1-m}}}\;\int_{\hnn}u^p(\tau,x)\;\G(x_0,x)\dg(x)\;.
\end{equation*}
Then we split the last integral in two parts: fix $R>0$ to be determined later and let $\varepsilon\in(0,p+m-1)$
\begin{equation*}\begin{split}
&u^{p+m-1}(t,x_0)
\leq \c_{p,m}\;\frac{t_1^{\frac{p+m-1}{1-m}}}{(t-\tau)^{\frac{p}{1-m}}}\;
\|u(\tau)\|_{\infty}^{p+m-1-\varepsilon}\int_{B_R(x_0)}u^{1-m+\varepsilon}(\tau,x)\;\G(x_0,x)\dg(x)\\
& + \c_{p,m}\;\frac{t_1^{\frac{p+m-1}{1-m}}}{(t-\tau)^{\frac{p}{1-m}}}\;\int_{\hnn\setminus B_R(x_0)}u^p(\tau,x)\;\G(x_0,x)\dg(x)\\
\leq&\frac{1}{2}\|u(\tau)\|_\infty^{p+m-1}\!\!+2^{\frac{p+m-1-\varepsilon}{\varepsilon}}\!\!
\left[\!\c_{p,m}\!\frac{t_1^{\frac{p+m-1}{1-m}}}{(t-\tau)^{\frac{p}{1-m}}}\!
\int_{B_R(x_0)}\!\!\!\!\!u^{1-m+\varepsilon}(\tau,x)\,\G(x_0,x)\dg(x)\right]^{\!\!\frac{p+m-1}{\varepsilon}}\\
& + \c_{p,m}\;\frac{t_1^{\frac{p+m-1}{1-m}}}{(t-\tau)^{\frac{p}{1-m}}}\;\int_{\hnn\setminus B_R(x_0)}u^p(\tau,x)\;\G(x_0,x)\dg(x)\\
\end{split}\end{equation*}
by using Young's inequality, $ab\leq\frac{1}{2}a^{\sigma}+2^{\frac{1}{\sigma-1}}\;b^{\frac{\sigma}{\sigma-1}}$, with $\sigma=\tfrac{p+m-1}{p+m-1-\varepsilon}>1$.

\noindent Taking supremum w.r.t. $x_0\in\hnn$ on both sides we obtain for all  $0\leq t_0\leq \tau< t \leq t_1$,	
\begin{equation*}\begin{split}
\|u(t)\|_\infty^{p+m-1}
\leq&\frac{1}{2}\|u(\tau)\|_\infty^{p+m-1}\\
& + 2^{\frac{p+m-1-\varepsilon}{\varepsilon}}
\left[\c_{p,m}\frac{t_1^{\frac{p+m-1}{1-m}}}{(t-\tau)^{\frac{p}{1-m}}}
\sup\limits_{\substack{\tau\in[t_0,t_1]\\x_0\in\hnn}} \int_{B_R(x_0)}u^{1-m+\varepsilon}(\tau,x)\,\G(x_0,x)\dg(x)\right]^{\frac{p+m-1}{\varepsilon}}\\
& + \c_{p,m}\;\frac{t_1^{\frac{p+m-1}{1-m}}}{(t-\tau)^{\frac{p}{1-m}}}\;\sup\limits_{\substack{\tau\in[t_0,t_1]\\x_0\in\hnn}}\int_{\hnn\setminus B_R(x_0)}u^p(\tau,x)\;\G(x_0,x)\dg(x)
\end{split}\end{equation*}
Finally, we can apply De Giorgi's Lemma, see \cite[Lemma 6.3]{BII}, with the function $Z(t):=\|u(t)\|^{p+m-1}_\infty$ and we obtain \eqref{Green.1.0}.

\noindent$\bullet~$\textsc{Step 2. } \textit{Proof of \eqref{Thm.Smoothing.ineq}  when $u_0\in L^1(\hnn)\cap L^{\infty}(\hnn)$}. \nero

The proof consists in estimating the two terms on the right-hand side of \eqref{Green.1.0}.

Fix $\varepsilon\in\left(0,\tfrac{2sp-N(1-m)}{N}\right)\subset(0,p+m-1)$ and
\[
1< \frac{p}{p-1+m-\varepsilon}=q<\frac{N}{N-2s} \quad \text{and}\quad 1<\frac{N}{2s}< q'= \frac{q}{q-1}=\frac{p}{1-m+\varepsilon}\,,
\]
where the above inequalities hold for every $p\ge 1$ if $m\in(m_c,1)$ and for $p> p_c$ if $m\in(0,m_c]$.
Then, by Hölder's inequality, the estimate \eqref{Hyp.Green.HN3} and the $L^p$-norm decay \eqref{p-decay}, we have
\begin{equation*}\begin{split}
\int_{B_R(x_0)}u^{1-m+\varepsilon}(\tau,x)\,\G(x_0,x)\dg(x)
&\le \|u(\tau)\|_{q'(1-m+\varepsilon)}^{1-m+\varepsilon} \left[\int_{B_R(x_0)}\left(\G(x_0,x)\right)^{q}\dg(x)\right]^{\frac{1}{q}}\\
&\le C \|u(t_0)\|_{p}^{1-m+\varepsilon} \,R^{\frac{N-q(N-2s)}{q}}
=C \|u(t_0)\|_{p}^{1-m+\varepsilon} \;{R^{2s-\frac{N}{q'}}}
\end{split}\end{equation*}
where   $C>0$ depends on $q,N,s$.

Next we estimate the second term on the right-hand side of \eqref{Green.1.0}. By \eqref{green-euc} and using the $L^p$-norm decay we have
\begin{equation*}\begin{split}
\int_{\hnn \setminus B_R(x_0)}u^p(\tau,x)&\;\G(x_0,x)\dg(x)
\le  C \int_{\hnn\setminus B_R(x_0)}\frac{u^p(\tau,x)}{r(x, x_0)^{N-2s}}\dg(x)\le C\frac{\|u(t_0)\|_p^p}{R^{N-2s}}\,.
\end{split}\end{equation*}

\noindent Plugging the above estimates in \eqref{Green.1.0} with $t_0=0$ and $t_1=t$, we obtain
\begin{equation}\label{R1}\begin{split}
\|u(t)\|_\infty^{p+m-1}
&\leq 2^{\frac{p+m-1-\varepsilon}{\varepsilon}}
c_1 \left[\frac{\|u_0\|_{p}^{1-m+\varepsilon}}{t} R^{\frac{2sp-N(1-m+\varepsilon)}{p}}\right]^{\frac{p+m-1}{\varepsilon}}\!\!\!\!
+ c_2\;\frac{\|u_0\|^p_{p}}{t}\;
\frac{ 1}{R^{N-2s}}\,,
\end{split}\end{equation}
where $c_1,c_2>0$ depend on $q,N,s$ ($c_1$ depends also on $\varepsilon$).
Finally, we choose
\[
R=\left(\frac{t}{\|u_0\|_{p}^{1-m}}\right)^{\frac{p}{2sp-N(1-m)}}
\]
and we obtain the desired smoothing effects \eqref{Thm.Smoothing.ineq} for every $p\ge 1$ if $m\in(m_c,1)$ and $p> p_c$ if $m\in(0,m_c]$.

\noindent$\bullet~$\textsc{Step 3. } \textit{Proof of \eqref{Thm.Smoothing.ineq} for WDS with  $u_0\in  L^1_{\GH}(\hnn) \cap L^p(\hnn)$ with with $p_c<p < +\infty$ if $m\in(0,m_c]$ or $1 \leq p < +\infty$ if $m\in(m_c,1)$.} We approximate the initial datum $0\leq u_0\in L^p(\hnn)$  by a monotone sequence of truncates $u_{0,n}=\min\lbrace u_0,n\rbrace  \chi_{B_n(o)} \in L^1(\hnn)\cap L^\infty(\hnn)$ so that $u_{0,n}\to u_0$ in $L^p$. Since $u_{0,n}\in L^1(\hnn)\cap L^{\infty}(\hnn) $, there exists a sequence of mild solutions $u_n(t)\in L^1(\hnn)\cap L^{\infty}(\hnn)$, which are WDS and satisfy the smoothing estimates \eqref{Thm.Smoothing.ineq} for every $p\ge 1$ if $m\in(m_c,1)$ and $p> p_c$ if $m\in(0,m_c]$. As a consequence, by lower semicontinuity of the $L^\infty$ norm we obtain
\begin{align*}
\|u(t)\|_\infty&\leq\lim\limits_{n\rightarrow\infty}\|u_n(t)\|_\infty\leq\lim\limits_{n\rightarrow\infty} \ka t^{-N\vartheta_{p}}\|u_{0,n}\|_p^{2sp\vartheta_{p}}=\ka t^{-N\vartheta_{p}}\|u_0\|_p^{2sp\vartheta_{p}}\,.
\end{align*}
\par

This concludes the proof of the smoothing effects of Theorem \ref{Thm.Smoothing} if $ \hn $ satisfies both Assumptions \ref{general} and \ref{ch}. When only Assumptions \ref{general} is satisfied, then Step 1 and Step 3 hold with basically no changes while, since we only have \eqref{Hyp.Green.HN1}, Step 2 holds provided that $R\leq 1$. This gives the proof of \eqref{Thm.Smoothing.ineq} for $0<t\leq \|u_0\|^{1-m}_{p} $. To prove \eqref{Thm.Smoothing.ineqweight0}, we fix $R=1$ in \eqref{R1} and we get

\begin{equation*}\begin{split}
\|u(t)\|_\infty^{p+m-1} \leq \frac{C}{t}  \| u_0 \|_{p}^p \left[  \left( \frac{\left\| u_0 \right\|_{p}^{(1-m)}}{t}\right)^{\frac{p+m-1-\varepsilon}{\varepsilon}} +1 \right]
 \end{split}
\end{equation*}
for some $C>0$ depending on $q,N,s$ and $\varepsilon\in\left(0,\tfrac{2sp-N(1-m)}{N}\right)\subset(0,p+m-1)$. Then \eqref{Thm.Smoothing.ineqweight0} follows by taking
$ \frac{\left\| u_0 \right\|_{p}^{1-m}}{t} \leq 1$. $\Box$

	\section{Proof of Theorem \ref{ext}} \label{proof2}

Let $ \hn $ satisfy Assumption \ref{general}. By exploiting the Stroock-Varopoulos inequality \eqref{Stroo} with $q=\frac{p-1+m}{m}>1$ (yielding $p>1$) and $v=u^m$, we get
\begin{align}\label{der.Lp.norm}
	\frac{\rd}{\dt}\int_{\hnn}u^p  \dg(x) &=-p\int_{\hnn}u^{p-1}(- \Delta_{\hnn})^{s} u^m  \dg(x) \leq -c_{m,p}\|(- \Delta_{\hnn})^{\frac{s}{2}} u^{\frac{p+m-1}{2}}\|_2^2
\end{align}
where $c_{m,p}:= \frac{4 m p (p-1)}{(p-1+m)^2}$. By taking $p=p_c$ in the above and using Sobolev inequality \eqref{Sobolev-frac} we deduce that
\begin{align}\label{ode}
	\frac{\rd}{\dt} \|u\|_{p_c}^{p_c} \leq -c_{m,p}\;C^{-1}\|u\|_{p_c}^{p_c+m-1}
\end{align}
and, in turn, that $\frac{\rd}{\dt} \|u\|_{p_c}^{1-m} \leq -c_{p}:=- c_{m,p}\;C^{-1}\, \frac{2s}{N}$. The latter inequality, integrated in the interval $[t_0,t_1]$, yields
$$
\|u(t_1)\|_p^{1-m} - \|u(t_0)\|_p^{1-m} \leq c_p (t_1-t_0)
$$
which both gives the existence of an extinction time $T=T(u_0)$ (by taking $t_0=0$ and $t_1=t$) and \eqref{Lp estimate} for $p=p_c$ (recall that $p_c>1$ only if $0<m<m_c$).

 Let now $ \hn $ also satisfy Assumption  \ref{neg}. If $p> p_c$, then $2 < \frac{2p}{p+m-1}< \frac{2N}{N-2s}$, whence, interpolating and exploiting Sobolev inequality \eqref{Sobolev-frac} and Poincar\'e inequality \eqref{poi}, we get $$\|w\|_{\frac{2p}{p+m-1}}\leq C \|(- \Delta_{\hnn})^{\frac{s}{2}} w\|_2 \qquad \text{for all } w\in H^s(\hnn) $$
for some $C>0$ only depending on $N,s$. Then, inserting this into \eqref{der.Lp.norm} with $w=u^{\frac{p+m-1}{2}}$ we get an inequality of the form \eqref{ode} and integrating on the interval $[0,t]$ we obtain \eqref{Lp estimate} for $p>p_c$. $\Box$	  		

\section{Proof of Theorem \ref{Thm.SmoothingPhi}} \label{proof3}
The proof follows by estimating the two terms on the right-hand side of \eqref{Green.1.0} for small and large times.  As explained in the proof of Theorem \ref{Thm.Smoothing}, it is enough to prove the estimate for bounded nonnegative mild solutions in $ L^1(\hnn)\cap L^{\infty}(\hnn)$.

\noindent$\bullet~$\textsc{proof of \eqref{Thm.Smoothing.ineqweight}.} Let $R\leq 1$. Arguing as in the proof of Theorem \ref{Thm.Smoothing}, by recalling the definition \eqref{normLGx0} of the norm $L^p_{x_0,\GH}$ and exploiting the $L^p_{x_0,\GH}$ stability given by Lemma \ref{LpPhi stability}, we estimate the first term of \eqref{Green.1.0} as follows:
\begin{equation} \label{first}
\begin{split}
\int_{B_R(x_0)}u^{1-m+\varepsilon}(\tau,x)\,\G(x_0,x)\dg(x)& \le C \|u(\tau)\|_{L^p(B_R(x_0))}^{1-m+\varepsilon} \,R^{\frac{N-q(N-2s)}{q}}\\
& \le C \left\| u_0 \right\|_{L^p_{x_0,\GH}}^{1-m+\varepsilon} \,R^{\frac{N-q(N-2s)}{q}}
\end{split}\end{equation}
where   $C>0$ depends on $q,N,s$ and $\varepsilon\in\left(0,\tfrac{2sp-N(1-m)}{N}\right)\subset(0,p+m-1)$.

On the other hand, by exploiting \eqref{green-euc}, we get the estimate of the second term on the right-hand side of \eqref{Green.1.0}:
\begin{equation*}
\begin{split}
&\int_{\hnn \setminus B_R(x_0)}u^p(\tau,x)\;\G(x_0,x)\dg(x)\\
&= \int_{B_1(x_0) \setminus B_R(x_0)}u^p(\tau,x)\;\G(x_0,x)\dg(x)+\int_{\hnn \setminus B_1(x_0)}u^p(\tau,x)\;\G(x_0,x)\dg(x)\\
&\le  \frac{C}{R^{N-2s}}\int_{B_1(x_0)}u^p(\tau,x)\dg(x)+ \frac{1}{R^{N-2s}} \int_{\hnn \setminus B_1(x_0)}u^p(\tau,x)\;\G(x_0,x)\dg(x)\\
&\le  \frac{C}{R^{N-2s}}  \| u_0 \|_{L^p_{x_0,\GH}}^{p} \,.
\end{split}\end{equation*}
Then, by the same argument of the proof of \eqref{Thm.Smoothing.ineq} and taking the supremum w.r.t. $x_0\in \hnn$, we get:
		$$
		  \|u(t)\|_\infty \leq \ka\;\frac{\|u_0\|_{L^p_{\GH}}^{2sp\,\vartheta_{p} }}{t^{N\vartheta_{p}}}\,,	
		 $$
 provided that
\[
R=\left(\frac{t}{\|u_0\|_{L^p_{\GH}}^{1-m}}\right)^{\frac{p}{2sp-N(1-m)}} \leq 1\,,
\]
namely for $t\leq \|u_0\|^{1-m}_{L^p_{\GH}} $, whence \eqref{Thm.Smoothing.ineqweight} follows.


\par

\noindent$\bullet~$\textsc{proof of \eqref{Thm.Smoothing.ineqweight2}.}
If $R=1$ formula \eqref{Green.1.0} combined with \eqref{first} yields
\begin{equation*}
\begin{split}
\|u(t)\|_\infty^{p+m-1} & \leq C \frac{1}{t^{\frac{p+m-1}{\varepsilon}}} \left\| u_0 \right\|_{L^p_{x_0,\GH}}^{(1-m+\varepsilon) {\frac{p+m-1}{\varepsilon}}} +C \frac{1}{t}\left\| u_0 \right\|_{L^p_{x_0,\GH}}^p
\\
& \leq \frac{C}{t}  \| u_0 \|_{L^p_{\GH}}^p \left[  \left( \frac{\left\| u_0 \right\|_{L^p_{\GH}}^{(1-m)}}{t}\right)^{\frac{p+m-1-\varepsilon}{\varepsilon}} +1 \right]
 \end{split}
\end{equation*}
where   $C>0$ depend on $q,N,s$ and $\varepsilon\in\left(0,\tfrac{2sp-N(1-m)}{N}\right)\subset(0,p+m-1)$. Then the thesis follows by taking
$ \frac{\left\| u_0 \right\|_{L^p_{\GH}}^{1-m}}{t} \leq 1\,.$   $ \Box$

\section{Proof of Theorem \ref{lowerth}}\label{proof4}

By testing \eqref{def_eq} with a smooth approximation of $\delta_{x_0}\chi_{[t_0, t_1]}$ and arguing as in the proof of \cite[Lemma 3.4]{BII}, for all $x \in M$ and $0\leq t_0\le t_1$ one has
	$$
	\frac{u^m(t_0,x)}{t_0^{\frac{m}{1-m}}} \geq  \frac{1}{1-m}\int_M  \frac{u(t_0,y)-u(t_1,y)}{ t_1^{\frac{1}{1-m}}-t_0^{\frac{1}{1-m}} }\, \GH(x,y)\dg \,.
	$$
Then, taking $t_0=t$ and $t_1=T$ in the above and recalling \eqref{green-est-low}, we get
\begin{equation}\begin{split}\label{pointwise Lp estimate_low}
	u^m(t,x) &\geq  \frac{t^{\frac{m}{1-m}}}{(1-m)\left(T^{\frac{1}{1-m}}-t^{\frac{1}{1-m}} \right)}\int_M u(t,y)\, \GH(x,y)\dg(y) \\ &
	\geq C  \frac{t^{\frac{m}{1-m}}}{(1-m)\left(T^{\frac{1}{1-m}}-t^{\frac{1}{1-m}} \right)} \,\left( \int_{B_1(x)} u (t,y)  \dg(y)\right.\\ &\left.+ \int_{\hn \setminus B_1(x)}  u(t,y)  \,\GM(x,y) \, \dg(y) \right) \\ &= C  \frac{t^{\frac{m}{1-m}}}{(1-m)\left(T^{\frac{1}{1-m}}-t^{\frac{1}{1-m}} \right)} \, \left\| u (t) \right\|_{L^1_{x,\GM}}  \,,
	\end{split}
\end{equation}
	for some $C=C(N,k,c,s)$. This completes the proof of \eqref{pointwise0}. \par
	
	In order to prove \eqref{pointwise}, we refine our estimate of the r.h.s. of the first line of formula \eqref{pointwise Lp estimate_low}. Recalling \eqref{FFE}, we formally compute
	$$\frac{d}{dt} \int_M u(t,y)\, \GH(x,y)\dg=- \int_M (- \Delta_{\hn})^{s} \!\left(u^m\right) \, \GH(x,y) \dg=-u^m(t,x)\geq -t^{-\alpha}M_0^m\,,$$
	where the latter inequality comes from \eqref{Thm.Smoothing.ineq} and $\alpha= \frac{Nm}{2sp-N(1-m)}\in (0,1)$ for $p>\frac{N}{2s}$, while $M_0^m=\ka^m\,\|u_0\|^{2sp\vartheta_{p}m}_p$. Then, integrating between $0$ and $t$, we get that
	$$\int_M u(t,y)\, \GH(x,y)\dg- \int_M u_0(y)\, \GH(x,y)\dg\geq -\frac{M_0^m}{1-\alpha} t^{1-\alpha}$$
	and, in turn, that
	$$\int_M u(t,y)\, \GH(x,y)\dg\geq \frac{1}{2} \int_M u_0(y)\, \GH(x,y)\dg \qquad \text{for } t\in[0,t_0(u_0)]$$
	for some $0<t_0(u_0)\leq T(u_0)$ sufficiently small. Regarding the above estimates, the finiteness of the last term can be checked by applying the H\"older inequality and exploiting the estimates \eqref{green-euc} and \eqref{Hyp.Green.HN1} as follows
   \begin{align*}
\int_M u_0(y)\, \GH(x,y) \dg&=\int_{B_1(o)} u_0(y)\, \GH(x,y) \dg+ \int_{M\setminus B_1(o)} u_0(y)\, \GH(x,y) \dg\\
&\leq \|u_0\|_{p}  \|\GH(x,y) \|_{L^{\frac{p}{p-1}}(B_1(o))}+C \|u_0\|_{1}\\
&\leq C( \|u_0\|_{p} \int_0^1 \frac{r^{N-1}}{r^{\frac{p(N-2s)}{p-1}}} dr+ \|u_0\|_{1}) \leq C( \|u_0\|_{p} + \|u_0\|_{1}) \,,
	\end{align*}
	where $\frac{p(N-2s)}{p-1}<N$ since $p>\frac{N}{2s}$ by assumption.

	  Finally, the same arguments of \cite[Lemma 3.2]{BBGM} (see Lemma \ref{LpPhi stability}), yield
	$$\int_M u_0(y)\, \GH(x,y)\dg \geq \left(1\wedge r(x_0,x)^{N-2s}\right)\GH(x,x_0)\,\|u_0\|_1$$
	and \eqref{pointwise} follows by combining the first line of formula \eqref{pointwise Lp estimate_low} with the above estimates. $ \Box$

\section{Appendix}
\begin{proof}[Proof of Theorem \ref{fracnash}]
It is known from \cite[Theorem 2.4.2]{DA}, via duality and the semigroup property, that \eqref{Nash} is equivalent to the bound
\begin{equation}\label{ultra}
\|T_t\|_{1,\infty}\le C t^{-\frac N2}\ \ \ \forall t>0,
\end{equation}
for a suitable $C>0$, where $T_t=e^{t\Delta}$ is the (non-fractional) heat semigroup and $\|\cdot\|_{1,\infty}$ is the operator norm from $L^1$ to $L^\infty$. Let us consider the fractional heat semigroup $T_t^{(s)}=e^{-t(-\Delta)^s}$, where the operator power is taken as usual in the spectral sense. It follows, see e.g. \cite[Section 4.3]{J} and references quoted, that, on a suitable core of functions:
\begin{equation}\label{subordinato}
T_t^{(s)}(u)=\int_0^{+\infty}\text{d}v\,T_v(u)\,g_t^{(s)}(v),
\end{equation}
where $g_t^{(\alpha)}>0$ is characterized by the fact that its (one-sided) Laplace transform is $e^{-tx^s}$, i.e.:
\[
L\left(g_t^{(s)}\right)(x):=\int_0^{+\infty}\text{d}v\,e^{-vx}\,g_t^{(s)}(v)=e^{-tx^s}\ \ \ \forall x\ge0,
\]
see again \cite[Section 3.9]{J} for further details on subordination of semigroups. It follows from \eqref{ultra} and \eqref{subordinato} that:
\begin{equation}\label{sub-bound}
\|T_t^{(s)}\|_{1,\infty}\le \int_0^{+\infty}\text{d}v\,\frac{g_t^{(s)}(v)}{v^{\frac N2}}.
\end{equation}
We shall now show that the latter integral is finite and determine its dependence on $t$. In order to do this, let us write, for any $\beta>0$:
\[\begin{aligned}
+\infty&>\int_0^{+\infty}x^{\beta-1}e^{-tx^s}\,\text{d}x=\int_0^{+\infty}\text{d}x\,x^{\beta-1}\int_0^{+\infty}\text{d}v\,e^{-xv}\,g_t^{(s)}(v)\\
&=\int_0^{+\infty}\text{d}v\,g_t^{(s)}(v)\int_0^{+\infty}\text{d}x\,x^{\beta-1}e^{-xv}\\
&=\int_0^{+\infty}\text{d}v\,g_t^{(s)}(v)\int_0^{+\infty}\frac{\text{d}y}{v}\left(\frac yv\right)^{\beta-1}e^{-y}\\
&=\int_0^{+\infty}\text{d}v\,\frac{g_t^{(s)}(v)}{v^\beta}\int_0^{+\infty}\text{d}y\,y^{\beta-1}e^{-y}\\
&=c_\beta^{-1}\int_0^{+\infty}\text{d}v\,\frac{g_t^{(s)}(v)}{v^\beta},
\end{aligned}
\]
where the second step follows by Fubini's Theorem, in the third one we set $y=xv$ and, finally, we set $c_\beta^{-1}:=\int_0^{+\infty}\text{d}y\,y^{\beta-1}e^{-y}$. Therefore,  $\int_0^{+\infty}\text{d}v\,\frac{g_t^{(s)}(v)}{v^\beta}<+\infty$ for all $\beta>0$ and we have:
\begin{equation}\label{power}
\int_0^{+\infty}\text{d}v\,\frac{g_t^{(s)}(v)}{v^\beta}=c_\beta\int_0^{+\infty}\text{d}x\,x^{\beta-1}e^{-tx^s}.
\end{equation}
Let us now consider the integral in the r.h.s. of \eqref{power}, that we can rewrite as follows, setting $x=yt^{-1/s}$:
\begin{equation}\label{num}\begin{aligned}
&\int_0^{+\infty}\text{d}x\,x^{\beta-1}e^{-tx^s}=\int_0^{+\infty}\frac{\text{d}y}{t^{\frac1s}}\,\left(\frac y{t^{\frac1s}}\right)^{\beta-1}e^{-y^s}\\\
&=\frac1{t^{\frac\beta s}}\int_0^{+\infty}\text{d}s\,y^{\beta-1}e^{-y^s}=\frac{\tilde c_\beta}{t^{\frac\beta s}},
\end{aligned}\end{equation}
where $\tilde c_\beta=\int_0^{+\infty}\text{d}s\,y^{\beta-1}e^{-y^s}$. Therefore, by \eqref{sub-bound}, \eqref{power} and \eqref{num} we have:
\begin{equation*}
\begin{aligned}
\|T_t^{(s)}\|_{1,\infty}\le \int_0^{+\infty}\text{d}v\,\frac{g_t^{(s)}(v)}{v^{\frac N2}}=c_{d/2}\int_0^{+\infty}\text{d}x\,x^{\frac d2-1}e^{-tx^s}=\frac{c_{d/2}\tilde c_{d/2}}{t^{\frac d{2s}}}=\frac{k}{t^{\frac d{2s}}}
\end{aligned}
\end{equation*}
for a suitable $k>0$ and for all $t>0$. Therefore, we can again apply \cite[Theorem 2.4.2]{DA} and get the validity of the fractional Nash inequality \eqref{Nash-frac}.

\end{proof}

\par\bigskip\noindent
\textbf{Acknowledgments.}
The first and third authors are members of the Gruppo Nazionale per l'Analisi Matematica, la Probabilit\`a e le loro Applicazioni (GNAMPA, Italy) of the Istituto Nazionale di Alta Matematica (INdAM, Italy). The research of the first author was carried out within the project: Geometric-Analytic Methods for PDEs and Applications (GAMPA), ref. 2022SLTHCE - funded by European Union - Next Generation EU within the PRIN 2022 program (D.D. 104 - 02/02/2022 Ministero dell'Universit\'a e della Ricerca).  The research of the third author was carried out within the project: Partial differential equations and related geometric-functional inequalities, ref. 20229M52AS - funded by European Union - Next Generation EU within the PRIN 2022 program (D.D. 104 - 02/02/2022), Ministero dell'Universit\'a e della Ricerca). This manuscript reflects only the authors' views and opinions and the Ministry cannot be considered responsible for them. \par  The authors are partially supported by the Projects PID2020-113596GB-I00 and PID2023-150166NB-I00 (Spanish Ministry of Science and Innovation). The second author acknowledges the financial support from the Spanish Ministry of Science and Innovation, through the ``Severo Ochoa Programme for Centres of Excellence in R\&D'' (CEX2019-000904-S and CEX2023-001347-S) and by the E.U. H2020 MSCA programme, grant agreement 777822.

%
%
%
%


\end{document}